%% file: datadrivenregulator.tex
\documentclass[journal]{IEEEtran}

\IEEEoverridecommandlockouts                              % This command is only needed if 
\usepackage{amsmath,amssymb,color,cite}%,showkeys}
\usepackage{graphicx}

\usepackage{latexsym}
\usepackage{verbatim}
\usepackage{cite}
\usepackage[english]{babel}
\usepackage[latin1]{inputenc}
\usepackage{enumerate}
\usepackage{amsmath,amsthm}

\usepackage{tikz}
\usetikzlibrary{tikzmark}
\usepackage{amsfonts}
\usepackage{amssymb}
\usepackage{amsbsy}
\usepackage{bm}

    \let\geq\geqslant

\newcommand{\bmu}{\bm{u}}
\newcommand{\bmx}{\bm{x}}

\newcommand{\bmd}{\bm{d}}
\newcommand{\bmr}{\bm{r}}
\newcommand{\bmz}{\bm{z}}

\input{ka-newcommands}

\newcounter{todocounter}
\usepackage[colorinlistoftodos]{todonotes}

\newtheorem{theorem}{Theorem}
\newtheorem{lemma}[theorem]{Lemma}

\newtheorem{prop}[theorem]{Proposition}
\theoremstyle{definition}
\newtheorem{definition}[theorem]{Definition}
\newtheorem{example}[theorem]{Example}

\newtheorem{conjecture}[theorem]{Conjecture}
\newtheorem{remark}[theorem]{Remark}

\title{An informativity approach to the data-driven algebraic regulator problem} 
\author{ Harry L. Trentelman, Henk J. van Waarde and 
M. Kanat Camlibel
	\thanks{The authors are with the Bernoulli Institute for Mathematics, Computer Science, and Artificial Intelligence, University of Groningen, Nij\-enborgh 9, 9747 AG, Groningen, The Netherlands. (email: {\footnotesize{\tt h.j.van.waarde@rug.nl; m.k.camlibel@rug.nl;h.l.trentelman@rug.nl}}).}
}

\begin{document}

\maketitle
\begin{abstract}
In this paper, the classical algebraic regulator problem is studied in a data-driven context. The endosystem is assumed to be an unknown system that is interconnected to a known exosystem that generates disturbances and reference signals. The problem is to design a regulator so that the output of the (unknown) endosystem tracks the reference signal, regardless of its initial state and the incoming disturbances. In order to do this, we assume that we have a set of input-state data on a finite time-interval. We introduce the notion of data informativity for regulator design, and establish necessary and sufficient conditions for a given set of data to be informative. Also, formulas for suitable regulators are given in terms of the data.  Our results are illustrated by means of two extended examples.
%The use of persistently exciting data has recently been popularized in the context of data-driven analysis and control. Such data have been used to assess system theoretic properties and to construct control laws, without using a system model. Persistency of excitation is a strong condition that also allows unique identification of the underlying dynamical system from the data within a given model class. In this paper, we develop a new framework in order to work with data that are {\em not\/} necessarily persistently exciting. Within this framework, we investigate necessary and sufficient conditions on the informati\-vity of data for several data-driven analysis and control problems. For certain analysis and design problems, our results reveal that persistency of excitation is not necessary. In fact, in these cases data-driven analysis/control is possible while the combination of (unique) system identification and model-based control is not. For certain other control problems, our results justify the use of persistently exciting data as data-driven control is possible only with data that are informative for system identification.

\end{abstract}

\section{Introduction}
Recently, the paradigm of data-driven control has gained a lot of attention in analysis and controller design of linear systems \cite{vanWaarde2020b,vanWaarde2020c, DePersis2020,vanWaarde2020,Berberich2020,Dai2018,Monshizadeh2020,Markovsky2008,Coulson2019,Baggio2020, Allibhoy2020,Tabuada2020,Rapisarda2011}. Instead of using an explicit mathematical model, the data-driven approach uses only data obtained from the unknown system for verifying its system theoretic properties and for constructing controllers. Recently, it was argued in \cite{vanWaarde2020} that the data-driven approach can also be useful in cases where the given data do not give sufficient information to identify the `true' model for the system, for example due to the fact that the data are not persistently exciting.
Indeed, in \cite{vanWaarde2020} the notion of {\em informativity} of data was introduced to cover situations in which a given set of data gives rise to a whole {\em family of system models} that are compatible with the data. In other words, situations in which it is impossible to distinguish between models on the basis of the given data. A set of data is called informative for a given system property if the property holds for all systems compatible with the data. In \cite{vanWaarde2020}, the notion of informativity was also developed in the context of controller design. In particular, conditions for informativity of data for the following control problems were given: state feedback stabilization, deadbeat control, linear quadratic optimal control, and stabilization by dynamic output feedback. Also, formulas (in terms of the data) were given to compute suitable controllers. 

The aim of the present paper is to extend the framework of informativity to the classical algebraic regulator problem (see e.g \cite{Davison1975,Francis1977,Francis1975,Isidori1990} and the textbooks \cite{Trentelman2001,Saberi2000}). This is the problem of finding a feedback controller (called a regulator) that makes the output of the controlled system track some a priori given reference signal, regardless of the disturbance input entering the system, and the initial state. In the context of the algebraic regulator problem, the relevant reference signals and disturbances (such as step functions, ramps or sinusoids) are signals that are generated as solutions of suitable autonomous linear systems. Given such reference signal and class of disturbance signals, one first constructs a suitable generating autonomous system (called the exosystem). Next, this exosystem is interconnected to the control system (called the endosystem), and a new output is defined as the difference between the original system output and the reference signal. A regulator should then be designed to make the output of the interconnection converge to zero for all disturbances and initial states.

In this paper, the `true' endosystem is assumed to be unknown, and therefore no mathematical model is available. Instead, we have collected data on the input, endosystem state, and exosystem state in the form of samples on a finite time-interval. The exosystem is assumed to be known, since this system models the reference signals and possible disturbance inputs. Also, the matrices in the output equations are assumed to be known, since these specify the design specification (namely the output that should converge to zero) on the controlled system.  A given set of data will then be called informative for regulator design if the data contain sufficient information to design a single regulator for the entire family of systems that are compatible with this set of data.  We will establish necessary and sufficient conditions for a given set of data to be informative for regulator design.  In particular, it will be shown how to replace the characteristic regulator equations by their data-driven counterparts, and to compute suitable regulators.

We note that data-driven regulator design was studied before in \cite{deCarolis2018} and \cite{Carnevale2017}, albeit from a rather different perspective. We also mention alternative methods that deal with tracking objectives, such as iterative feedback tuning (IFT) and virtual reference feedback tuning (VRFT) as developed in \cite{Hjalmarsson1998} and \cite{Campi2002}, respectively. These methods do however not address the classical regulator problem, and are thus quite different from the work that will be presented in this paper.

The main contributions of the present paper are the following.
\begin{enumerate}
\item
We give a definition of the problem of data-driven tracking and regulation using the concept of informativity.
\item
We give necessary and sufficient conditions for data to be informative for regulator design, i.e., for the existence of a single regulator for all systems compatible with the given data.
\item
We establish formulas for computing these regulators, entirely in terms of the data.
\end{enumerate}
It should be noted that these regulators may be called {\em robust}, in the sense that a single regulator works for the whole set of systems that are compatible with the given data, see also \cite{deCarolis2018}.
%We will also illustrate the theory developed in this paper by means of two extended examples. 

The outline of this paper is as follows. In Section II, we illustrate the data-driven problem of tracking and regulation using an extended example. Subsequently we put the problem in a general framework, and define the concept of informativity for regulator design. In Section III, we review some classical basic material on the regulator problem. Then, in Section IV we formulate our main result, giving necessary and sufficient conditions for informativity for regulator design, and formulas to compute regulators. The main result is illustrated by means of two extended examples. Finally, in Section V, we formulate our conclusions.

\section{Data-driven tracking and regulation}\label{sec:prob}
We will first illustrate the problem to be considered in this paper by means of an extended example. 
\begin{example} \label{ex: example}
Consider the scalar linear time-invariant discrete-time system
\begin{equation}\label{e:ex}
\bmx(t+1)=a_s\bmx(t)+b_s\bmu(t) + \bmd(t),
\end{equation}
where $\bmx$ is the state, $\bmu$ the control input, and $\bmd$ a disturbance input. The values of $a_s$ and $b_s$ in this system representation are unknown. We assume that the disturbance can be any constant signal of finite amplitude. Suppose that we want the state $\bmx(t)$ to track the given reference signal $\bmr(t) = \cos \frac{\pi}{2}t$, for any constant disturbance input, regardless of the initial state of the system.  We want to design a control law for \eqref{e:ex} that achieves this specification. We assume that $\bmr$, $\bmx$ and $\bmd$ are available for feedback and allow control laws of the form 
\begin{equation} \label{e:law}
\bmu(t)  = k_1\bmr(t) + k_2 \bmr(t+1) + k_3 \bmd(t) + k_4 \bmx(t).
\end{equation}
Interconnecting \eqref{e:ex} and \eqref{e:law} results in the controlled system
\begin{equation*}
\begin{split}
\bmx(t+1) = (a_s +b_sk_4)\bmx(t) +(b_sk_3 +1)\bmd(t) + \\  b_sk_1\bmr(t) + b_sk_2\bmr(t+1),
\end{split}
\end{equation*}
where the gains $k_i$ should be designed such that $\bmx(t) - \bmr(t) \rightarrow 0$ as $t \rightarrow \infty$ for any constant disturbance input $\bmd$ and initial state $\bmx(0)$. It is also required that the controlled system is internally stable,  in the sense that $a_s + b_s k_4$ is stable \footnote{We say that a matrix is {\em stable\/} if all its eigenvalues are contained in the open unit disk.}. 

The values of $a_s$ and $b_s$ that represent the true system are unknown, but in the data-driven context it is assumed that we do have access to certain data. In particular, it is assumed that we have finite sequences of samples of $\bmx(t)$, $\bmu(t)$ and $\bmd(t)$ on a given time interval $\{0, 1, \ldots, \tau\}$, given by 
\begin{subequations}\label{eq:UXdata}
\begin{align}
U_-& := \bbm u(0) & u(1) & \cdots & u(\tau-1)\ebm, \\
X  & := \bbm x(0) & x(1) & \cdots & x(\tau)\ebm, \\
D_-& := \bbm d(0) & d(1) & \cdots & d(\tau-1)\ebm,
\label{eq: UXdata2}
\end{align}
\end{subequations}
where, in this particular example, by assumption $d(t) = d(0)$ for $t = 1,2, \ldots \tau-1$.  Define
\begin{align*}
X_+ & := \bbm x(1) & x(2) & \cdots & x(\tau) \ebm, \\
X_-  &:= \bbm x(0) & x(1) & \cdots & x(\tau-1) \ebm.
\end{align*}
It is assumed that these data are generated by the true system, so we must have
$ 
X_+=a_sX_-+b_sU_- +D_- .
$ 
For this example, the problem of data-driven control design is now to use the data \eqref{eq:UXdata} to determine whether a suitable controller \eqref{e:law} exists, and to compute the associated gains $k_1, k_2, k_3$ and $k_4$
using only these data.

Note that in the above, both the reference signal and the disturbance signals are generated by the  autonomous linear system
\begin{equation} \label{e:example-exo}
\begin{bmatrix}
\bmr_1(t+1) \\ \bmr_2(t+1) \\ \bmd(t+1)
\end{bmatrix} =
\begin{bmatrix} 0 & 1  &0 \\ -1 & 0 & 0 \\ 0 & 0 & 1
\end{bmatrix}
\begin{bmatrix}
\bmr_1(t) \\ \bmr_2(t) \\\bmd(t)
\end{bmatrix}
\end{equation}
with initial state $\bmr_1(0) = 1$ and $\bmr_2(0) = 0$, and $\bmd(0)$ arbitrary. Indeed, it can be seen that the reference signal $\bmr(t)=  \cos \frac{\pi}{2}t$ is equal to $\bmr_1(t)$. In addition, the solutions $\bmd(t)$ are all constant signals of finite amplitude. The autonomous system \eqref{e:example-exo} is called the {\em exosystem}.

The interconnection of the (unknown) to be controlled system \eqref{e:ex} (called the {\em endosystem}) with the exosystem \eqref{e:example-exo}, is represented by
\begin{equation} \label{e:exo-endo}
\begin{bmatrix} \bmr_1(t+1) \\ \bmr_2(t +1) \\ 
\bmd(t+1)  \\ \bmx(t+1) 
\end{bmatrix} 
\begin{bmatrix}
0 & 1 & 0 & 0 \\
-1 & 0 & 0 & 0 \\
0 & 0 & 1 & 0 \\
0& 0 & 1 & a_s
\end{bmatrix}
\begin{bmatrix} \bmr_1(t) \\ \bmr_2(t) \\ 
\bmd(t)  \\ \bmx(t) 
\end{bmatrix} 
+ \begin{bmatrix}
0 \\ 0\\ 0 \\ b_s
\end{bmatrix}
\bmu(t).
\end{equation}
In this representation, the part corresponding to the exosystem is known, but the part corresponding to the endosystem (specifically: $a_s$ and $b_s$) is unknown. We now also specify a (known) output equation 
$$
\bmz(t) =
\begin{bmatrix} 1 & 0 & 0 & -1
\end{bmatrix}
\begin{bmatrix} \bmr_1(t) \\ \bmr_2(t) \\ 
\bmd(t)  \\ \bmx(t)
\end{bmatrix}.
$$
Then the problem of our example can be rephrased as:  design a full {\em state feedback} control law \begin{equation*} 
\bmu(t) = k_1\bm r_1(t) +k_2 \bmr_2(t) + k_3 \bmd(t) + k_4 \bmx(t)
\end{equation*}
for the system \eqref{e:exo-endo}
such that in the controlled system we have $\bmz(t) \rightarrow 0$ as $t \rightarrow \infty$ for the initial states 
$\bmr_1(0) = 1$, $\bmr_2(0) = 0$, and $\bmd(0)$ arbitrary, while internal stability is achieved in the sense that 
$a_s + b_s k_4$ is a stable matrix. In order to allow tracking of signals from the richer class of all reference signals of the form $\bmr(t) = A \cos(\frac{1}{2} \pi t+ \omega)$ ($A$ and $\omega$ are determined by the initial states 
$\bmr_1(0) = 1$ and $\bmr_2(0)$), we may slightly relax the problem formulation and require $\bmz(t) \rightarrow 0$ as $t \rightarrow \infty$ for {\em all} initial states $\bmr_1(0), \bmr_2(0)$ and $\bmd(0)$. 
%If this property holds, and $a_s + b_s k_4$ is a stable matrix, the state feedback control law \eqref{e:sf} is called a {\em regulator}.
\end{example}

After having introduced our problem set up by means of the above example, we will now formulate it in a general framework. 

Consider an endosystem represented by
\begin{equation} \label{e:endo-general}
\bmx_2(t+1) = A_{2s}\bmx(t) + B_{2s}\bmu(t) + A_3\bmx_1(t).
\end{equation}
Here, $\bmx_2$ is the $n_2$-dimensional state, $\bmu$ the $m$-dimensional input, and $\bmx_1$ the $n_1$-dimensional state of the exosystem 
\begin{equation} \label{e:exo}
\bmx_1(t+1) = A_1\bmx_1(t).
\end{equation}
that generates all possible reference signals and disturbance inputs. The matrices $A_{2s}$ and $B_{2s}$ are unknown, but the matrix $A_1$ is known. Also $A_3$ is a known matrix that represents how the endosystem interconnects with the exosystem. 
The output to be regulated is specified by
\begin{equation} \label{e:outequation}
\bmz(t) = D_1 \bmx_1(t) + D_2 \bmx_2(t) + E \bmu(t),
\end{equation}
where the matrices $D_1,D_2$ and $E$ are known. By interconnecting the endosystem with the  state feedback controller
\begin{equation} \label{e:feedback}
\bmu(t) = K_1 \bmx_1(t) + K_2 \bmx_2(t), 
\end{equation}
we obtain the controlled system 
%\begin{subequations} 
%\begin{align*} %\label{e:controlled-general}
$$
\begin{bmatrix} \bmx_1(t+1) \\ \bmx_2(t+1)
\end{bmatrix}
   = 
\begin{bmatrix} A_1   & 0 \\ A_3 + B_2 K_1& A_{2s} + B_{2s} K_2
\end{bmatrix}
\begin{bmatrix} \bmx_1(t) \\ \bmx_2(t) 
\end{bmatrix},
$$
$$
 \bmz(t)  =  ~(D_1 + E K_1) \bmx_1(t) + (D_2  + E K_2)\bmx_2(t).
 $$
%\end{align*}
%\end{subequations}
If $\bmz(t) \rightarrow 0$ as $t \rightarrow \infty$ for all initial states $\bmx_1(0)$ and $\bmx_2(0)$, we say that the controlled system is {\em output regulated}.  If $A_{2s} + B_{2s} K_2$ is a stable matrix we call the controlled system {\em endo-stable}. If the control law \eqref{e:feedback} makes the controlled system both output regulated and endo-stable, we call it a {\em regulator}.

As illustrated in the example above, we assume that we do not know the true endosystem \eqref{e:endo-general}, and therefore the design of a regulator can only be based on available data. In the general framework, these are finite sequences of samples of $\bmx_1(t), \bmx_2(t)$ and $\bmu(t)$ on a given time interval $\{0, 1, \ldots, \tau\}$ given by 
%\begin{subequations}\label{e:UX1X2data}
\begin{align*}
U_-& := \bbm u(0) & u(1) & \cdots & u(\tau-1)\ebm, \\
X_{1-}  & := \bbm x_1(0) & x_1(1) & \cdots & x_1(\tau-1)\ebm, \\
X_2  & := \bbm x_2(0) & x_2(1) & \cdots & x_2(\tau)\ebm.
\end{align*}
%\end{subequations}
An endosystem with (unknown) system matrices $(A_2,B_2)$ is called {\em compatible} with these data if $A_2$ and $B_2$ satisfy the equation
\vspace{-2mm}
\begin{equation}  \label{e:data}
X_{2+} = A_{2} X_{2-} + A_3 X_{1-} + B_{2} U_- , 
\end{equation}
\vspace{-2mm}
where we denote
%\begin{subequations}\label{eq: def of X1- X1+}
\begin{align*}
X_{2-}& := \bbm x_2(0) & x_2(1) & \cdots & x_2(\tau -1) \ebm, \\
X_{2+}& := \bbm x_2(1) & x_2(2) & \cdots & x_2(\tau) \ebm.
%X_{1-} & := \bbm x_1(0) & x_1(1) & \cdots & x_1(T-1) \ebm
\end{align*}	
%\end{subequations}
The set of all $(A_2,B_2)$ that are compatible with the data is denoted by $\Sigma_\calD$, i.e.,
\begin{equation} \label{e:SigmaD}
\Sigma_\calD := \left\{ (A_2,B_2) \mid  \mbox{\eqref{e:data} holds} \right\}.
\end{equation}
We assume that the true endosystem $(A_{2s},B_{2s})$ is in $\Sigma_\calD$, i.e. the true system is compatible with the data.  In general, the equation \eqref{e:data} does not specify the true system uniquely, and many endosystems $(A_2,B_2)$ may be compatible with the same data.

Now we turn to controller  design based on the data $(U_-,X_{1-},X_2)$. Note that, since on the basis of the given data we can not distinguish between
the true endosystem and any other endosystem compatible with these data, a controller will be a regulator for the true system only if it is a regulator for any system with $(A_2,B_2)$ in $\Sigma_\calD$. If such regulator exists, we call the data {\em informative for regulator design}. More precisely:
\begin{definition}\label{def:informativity-regulator}
	We say that the data $(U_-,X_{1-},X_2)$ are informative for regulator design if there exists $K_1$ and $K_2$ such that the control law $\bmu(t) = K_1 \bmx_1(t) + K_2 \bmx_2(t)$ is a regulator for any endosystem with $(A_2,B_2)$ in $\Sigma_\calD$. 
\end{definition}

The problem that will be considered in this paper is to find necessary and sufficient conditions on the data $(U_-,X_{1-},X_2)$ to be informative for regulator design. Also, in case that these conditions are satisfied, we will explain how to compute a regulator using only these data. Before addressing this problem, in the next section we will review some basic material on the regulator problem.

\section{The regulator problem}\label{sec:regulator-problem}
In this section, we briefly review some basic material on the regulator problem. Following \cite{Trentelman2001}, we distinguish between analysis and design. 

We first consider the analysis question under what conditions a controlled system is endo-stable and output regulated. 
Consider the autonomous linear system represented by 
%\begin{subequations} \label{e:autsys}
\begin{align} \label{e:autsys}
\bmx_1(t+1)  &= A_1 \bmx_1(t), \nonumber \\
\bmx_2(t+1)  &= A_{2}\bmx_2(t) + A_3\bmx_1(t),
\end{align}
$$
\bmz(t) = D_1 \bmx_1(t) + D_2 \bmx_2(t).
$$
%\end{subequations}
In accordance with the terminology introduced in Section \ref{sec:prob}, we call this system endo-stable if $A_2$ is a stable matrix. We call it output regulated if $\bmz(t) \rightarrow 0$ as $t \rightarrow \infty$ for all initial states $\bmx_1(0)$ and $\bmx_2(0)$. The following is the discrete-time version of Lemma 9.1 in  \cite{Trentelman2001}
\begin{prop} \label{th:autonomous}
Assume that $A_1$ is anti-stable \footnote{We say that a matrix is {\em anti-stable\/} if all its eigenvalues $\lambda$ satisfy $|\lambda| \geq 1$}. Then the system \eqref{e:autsys} is endo-stable and output regulated if and only if  $A_2$ is stable and there exists a matrix $T$ satisfying the equations
\begin{equation}\label{e:aut-eq}
TA_1 - A_2 T  = A_3 ,~
D_1 + D_2T   = 0.
\end{equation}
In this case, $T$ is unique.
\end{prop}
Next, we consider the design problem and review conditions under which, for a given interconnection of an endosystem and exosystem, there exists a regulator, i.e., a controller that makes the controlled system endo-stable and output regulated. 
For the endosystem $\bmx_2(t+1)  = A_{2}\bmx_2(t) + B_2 \bmu(t) + A_3\bmx_1(t)$ together with the exosystem \eqref{e:exo} and output equation \eqref{e:outequation}, 
%Consider the to be controlled system 
%%\begin{subequations} \label{e:autsys}
%\begin{align*} 
%\bmx_1(t+1)  &= A_1 \bmx_1(t), \\
%\bmx_2(t+1) & = A_{2}\bmx_2(t) + B_2 \bmu(t) + A_3\bmx_1(t), 
%\end{align*}
%$$
%\bmz(t) = D_1 \bmx_1(t) + D_2 \bmx_2(t) + E\bmu(t).
%$$
%%\end{subequations}
the following is well-known and can be proven easily by extending results from \cite{Trentelman2001} to the discrete-time case:
\begin{prop} \label{th:existence}
Assume that $A_1$ is anti-stable. There exists a regulator of the form \eqref{e:feedback} if and only if $(A_2,B_2)$ is stabilizable and there exist matrices $T$ and $V$ satisfying the regulator equations
\begin{equation} \label{e:reg-eq}
TA_1 - A_2 T - B_2 V = A_3 , ~
D_1 + D_2T + EV = 0.
\end{equation}
In this case, a regulator is obtained as follows: choose any $K_2$ such that $A_2 + B_2 K_2$ is stable, and define $K_1 := -K_2 T + V$.
\end{prop}

\vspace{-4mm}
\section{The data-driven regulator problem}\label{sec:mainresult}
Clearly, a necessary condition for the data $(U_-,X_{1-},X_2)$ to be informative for regulator design is that they are informative for endo-stabilization:
\begin{definition} \label{e: endostab}
We call the data $(U_-,X_{1-},X_2)$ are \textit{informative} for endo-stabilization if there exists $K_2$ such that $A_2 + B_2 K_2$ is a stable matrix for all $(A_2,B_2)$ in $\Sigma_\calD$.
\end{definition}
In order to obtain necessary and sufficient conditions for informativity for endo-stabilization we formulate:
\begin{prop} \label{prop:endo-basic}
Let $\tau$ be a positive integer. Let $Z, X$ be real $n \times \tau$ matrices and let $U$ be a real $m \times \tau$ matrix. Consider the set
$
\Sigma_{(Z,X,U)} : = \{ (A,B) \mid Z = AX +BU \}.
$
Then the following hold:
\begin{enumerate}
\item
There exists a matrix $K$ such that $A + BK$ is stable for all $(A,B) \in \Sigma_{(Z,X,U)}$ if and only if $X$ has full row rank, and there exists a right-inverse $X^{\dagger}$ such that $Z X^{\dagger}$ is stable. In that case, by taking $K := U X^{\dagger}$ we have $A + BK$ is stable for all $(A,B) \in \Sigma_{(Z,X,U)}$. 
\item
For any $K$ such that $A + BK$ is stable  for all $(A,B) \in \Sigma_{(Z,X,U)}$ there exists a right-inverse $ X^{\dagger}$ such that $K = U X^{\dagger}$, and, moreover, $A + BK = Z X^{\dagger}$ for all $(A,B) \in \Sigma_{(Z,X,U)}$. 
\end{enumerate}
\end{prop}
\begin{proof}
The proof can be given by slightly adapting the proof of Theorem 16 in \cite{vanWaarde2020}.
\end{proof}
This immediately gives the following conditions for informativity for endo-stabilization.
\begin{lemma} \label{l:endo-spec}
The data $(U_-,X_{1-},X_2)$ are informative for endo-stabilization if and only if $X_{2-}$ has full row rank, and there exists a right inverse $X_{2-}^{\dagger}$ of $X_{2-}$ such that $(X_{2+} - A_3 X_{1-}) X_{2-}^{\dagger}$ is stable. In that case, by taking $K_2 := U_- X_{2-}^{\dagger}$ we have $A_2 + B_2 K_2$ is stable for all $(A_2,B_2) \in \Sigma_{\calD}$.
\end{lemma}

The following theorem is the main result of this paper. It gives necessary and sufficient conditions on the data to be  informative for regulator design, and explains how suitable regulators are computed using only these data.
\begin{theorem} \label{th:mainresult}
Assume that $A_1$ is anti-stable and suppose, for simplicity, that is diagonalizable. Then the data 
$(U_-,X_{1-},X_2)$ are informative for regulator design if and only if at least one of the following two conditions hold~\footnote{We denote by $\im M$ the image of the matrix $M$}:
\begin{enumerate}
\item
$X_{2-}$ has full row rank, and there exists a right-inverse $X_{2-}^\dagger$ of $X_{2-}$ such that $(X_{2+} - A_3 X_{1-}) X_{2-}^\dagger$ is stable and $D_2 +E U_-X_{2-}^\dagger =0$.  Moreover, $\im D_1 \subseteq \im E$. In this case, a regulator is found as follows: choose $K_1$ such that $D_1 + EK_1 =0$ and define $K_2 := U_- X_{2-}^\dagger$.
\item
$X_{2-}$ has full row rank and there exists a right-inverse $X_{2-}^\dagger$ of $X_{2-}$ such that $(X_{2+} - A_3 X_{1-}) X_{2-}^\dagger$ is stable. Moreover, there exists a solution $W$ to the linear equations
\begin{subequations} \label{e:W}
\begin{align}
X_{2-}W A_1 - (X_{2+} - A_3 X_{1-}) W = A_3 ,\\
D_1 +(D_2 X_{2-} + E U_-) W  = 0,
\end{align} 
\end{subequations}
In this case, a regulator is found as follows: choose $K_1 := U_- (I - X_{2-}^\dagger X_{2-}) W$ and $K_2 := U_- X_{2-}^\dagger$.
\end{enumerate}
\end{theorem}
Before turning to the proof, we will explain how to apply this theorem. What we know about the system are the system matrices $A_1, A_3, D_1,D_2$ and $E$ and the data $(U_-,X_{1-},X_2)$.
The aim is to use this knowledge to compute a {\em single} regulator $(K_1,K_2)$ that works {\em for all} endosystems $(A_2,B_2)$ in the set $\Sigma_D$ defined by \eqref{e:SigmaD}.

In order to check the existence of such regulator, we verify  the two conditions 1) and 2) in Theorem \ref{th:mainresult}. If neither of the two conditions holds, then the data are not informative.  On the other hand, if condition 1) holds then a regulator $(K_1,K_2)$ is computed as follows: 
\begin{itemize}
\item
find a right-inverse $X_{2-}^{\dagger}$ of $X_{2-}$ such that the matrix $(X_{2+} - A_3 X_{1-}) X_{2-}^\dagger$ is stable and $D_2 +E U_-X_{2-}^\dagger =0$,
\item
compute $K_1$ as a solution of $D_1 + EK_1 =0$,
\item
define $K_2 := U_- X_{2-}^\dagger$. 
\end{itemize}
If condition 2) holds then a regulator is computed as follows: 
\begin{itemize}
\item
find a right-inverse $X_{2-}^{\dagger}$ of $X_{2-}$ such that the matrix $(X_{2+} - A_3 X_{1-}) X_{2-}^\dagger$ is stable,
\item
find a solution $W$ of the data-driven regulator equations \eqref{e:W},
\item
define $K_1 := U_- (I - X_{2-}^\dagger X_{2-}) W$,
\item
define $K_2 := U_- X_{2-}^\dagger$.
\end{itemize}

\begin{proof}
($\Rightarrow$) We first prove sufficiency. Assume that the condition 1) holds. Since $(X_{2+} - A_3 X_{1-}) X_{2-}^\dagger$ is stable, the data are informative for endo-stabilization and by taking $K_2 := U_- X_{2-}^\dagger$ we have $A_2 + B_2 K_2$ is stable for all $(A_2,B_2) \in \Sigma_{\calD}$. Since $A_1$ is assumed to be anti-stable, this implies that for all $(A_2,B_2) \in \Sigma_{\calD}$ there exists a unique solution $T$ to the Sylvester equation 
$
TA_1 -( A_2 +B_2K_2)T = A_3 +B_2K_1.
$ 
By the fact that $D_1 + EK_1 = 0$ and $D_2 + EK_2 =0$, this solution $T$ also satisfies
$
D_1 + EK_1 + (D_2 +EK_2)T = 0.
$
Thus, for all $(A_2,B_2) \in \Sigma_\calD$, there exists a  matrix $T$ that satisfies the equations \eqref{e:aut-eq}. It follows from Proposition \ref{th:autonomous} that for all $(A_2,B_2) \in \Sigma_\calD$ the controlled system is endo-stable and output regulated.

Next, assume that condition 2) holds. By Lemma \ref{l:endo-spec}, the data are informative for endo-stabilization and by taking $K_2 : = U_1 X_{-}^{\dagger}$ we have $A_2 + B_2 K_2$ stable for all $(A_2,B_2) \in \Sigma_{\calD}$. Let $W$ satisfy the equations \eqref{e:W}. Define $T := X_{2-} W$ and $V : = U_- W$. Then the pair $(T,V)$ satisfies the regulator equations \eqref{e:reg-eq} for all $(A_2,B_2) \in \Sigma_\calD$. Then,  by Proposition \ref{th:existence}, for each such $(A_2,B_2)$ a regulator is given by the pair $(K_1, K_2)$, with $K_1 = -K_2T + V= -K_2X_{2-}W + U_-W = U_- (I - X_{2-}^{\dagger} X_{2-} ) W$. This completes the proof of the sufficiency part.

We will now turn to the necessity part. Assume that the data are informative for regulator design. By Proposition \ref{th:autonomous}, there exist $K_1$ and $K_2$ and for any $(A_2,B_2) \in \Sigma_{\calD}$ a matrix $T_{(A_2,B_2)}$ such that $A_2 + B_2 K_2$ is stable and 
%\begin{subequations}  \label{e:helpfull}
\begin{align*}
T_{(A_2,B_2)} A_1 - (A_2 + B_2K_2) T_{(A_2,B_2)}  &=  A_3 + B_2 K_1, \\
D_1 + EK_1 + (D_2 + EK_2) T_{(A_2,B_2)} & =    0.
\end{align*}
%\end{subequations}
We emphasize that $T_{(A_2,B_2)}$ may depend on the choice of  $(A_2,B_2) \in \Sigma_{\calD}$. 
However, since $A_2 + B_2 K_2$ is stable for all  $(A_2,B_2) \in \Sigma_{\calD}$, by Proposition \ref{prop:endo-basic} 
there exists a right-inverse $X_{2-}^{\dagger}$ of $X_{2-}$ such that $A_2 + B_2 K_2 = (X_{2+} - A_3 X_{1-}) X_{2-}^{\dagger}$ for all $(A_2,B_2) \in \Sigma_{\calD}$.  The latter matrix is independent of $(A_2,B_2)$. Call it $M$.
Define
$$
\Sigma_{\calD}^0 : = \{ (A_0,B_0) \mid \begin{bmatrix} A_0 & B_0 \end{bmatrix} \begin{bmatrix} X_{2-} \\ U_- \end{bmatrix} = 0 \}.
$$
 Note that $\Sigma_{\calD}^0$ is the solution space of the homogeneous version
 of  the defining equation \eqref{e:data} for $\Sigma_{\calD}$ (see \eqref{e:SigmaD}). 
We now distinguish two cases, namely (i) $B_0 K_1 = 0$ for all $(A_0,B_0) \in \Sigma_{\calD}^0$,  and (ii) $B_0 K_1 \neq  0$ for some $(A_0,B_0) \in \Sigma_{\calD}^0$.

First consider case (i). Then for all $(A_2,B_2), (\bar{A}_2,\bar{B}_2) \in \Sigma_{\calD}$ we have $B_2K_1 = \bar{B}_2 K_1$.  Thus, there exists a {\em common} matrix $T$  that solves the equations
\begin{align*}
T  A_1 - M T  =  A_3 + B_2 K_1, \\
D_1 + EK_1 + (D_2 + EK_2) T  =  0,
\end{align*}
for all $(A_2,B_2) \in \Sigma_{\calD}$. From this, we obtain 
$$
T A_1 - \begin{bmatrix} A_2 & B_2 \end{bmatrix} \begin{bmatrix} T \\ K_2 T + K_1 \end{bmatrix} = A_3
$$
for all $(A_2,B_2) \in \Sigma_{\calD}$, and therefore
$$
\begin{bmatrix} A_0 & B_0 \end{bmatrix} \begin{bmatrix} T \\ K_2 T + K_1 \end{bmatrix} = 0
$$
for all $(A_0,B_0) \in \Sigma_{\calD}^0$. This implies
$$
\im \begin{bmatrix} T \\ K_2 T + K_1 \end{bmatrix} \subseteq \im \begin{bmatrix} X_{2-} \\ U_- \end{bmatrix}.
$$
As a consequence, there exists a matrix $W$ such that 
$$ 
\begin{bmatrix} T \\ K_2 T + K_1 \end{bmatrix} = \begin{bmatrix} X_{2-} \\ U_- \end{bmatrix} W.
$$
Clearly, $W$ satisfies the equations \eqref{e:W}, showing that condition 2) holds.

Next, consider case (ii). Let $S$ be a real $(n_2 + m) \times r$ matrix such that \footnote{We denote by $\ker M$ the kernel of the matrix $M$}
$$
\ker \bbm X_{2-} \\ U_- \ebm^T =  \im S.
$$
Partition $S = \bbm S_1 \\ S_2 \ebm$. Then $(A_0,B_0) \in \Sigma_{\calD}^0$ if and only if $A_0 = NS_1^T$ and $B_0 = NS_2^T$ for some $n_2 \times r$ matrix $N$. Note that, by hypothesis, $S_2^T K_1 \neq 0$. 

Let $(A_2, B_2) \in \Sigma_{\calD}$. Recall that for any such $(A_2, B_2)$ there exists a unique $T_{(A_2,B_2)}$ such that 
\begin{align}
T_{(A_2,B_2)} A_1 - M T_{(A_2,B_2)}  =  A_3 + B_2 K_1 \nonumber \\
D_1 + EK_1 + (D_2 + EK_2) T_{(A_2,B_2)} =    0 \label{e:output}
\end{align}
Now let $N$ be any real $n_2 \times r$ matrix. Then also $(A_2+ NS_1^T, B_2 + NS_2^T) \in \Sigma_{\calD}$.
Define $T_N := T_{(A_2,B_2)} - T_{(A_2 + NS_1^T, B_2+ NS_2^T)}$. Then clearly $T_N$ is the unique solution to 
\begin{equation}  \label{e:Syl}
T_N A_1 - MT_N = NS_2^TK_1,
\end{equation}
which in addition satisfies $(D_2 + EK_2)T_N = 0$. 
Consider now a spectral decomposition $A_1 = Q^{-1} \Lambda Q$, where $\Lambda$ is the diagonal matrix $\Lambda = \diag(\lambda_1, \ldots \lambda_{n_1})$ and 
$$
Q = \bbm q_1 \\ \vdots \\ q_{n_1} \ebm, ~ Q^{-1} = \bbm \hat{q}_1 & \ldots & \hat{q}_{n_1} \ebm.
$$
Then, for fixed $N$, the unique solution $T_N$ to the Sylvester equation \eqref{e:Syl} can be expressed as
$$
T_N = \sum_{i =1}^{n_1} (\lambda_i I - M)^{-1} N S_2^T K_1 \hat{q}_i q_i
$$
(see \cite{Antoulas2005}),  which implies that $T_N Q^{-1}$ is equal to 
%\begin{eqnarray*}
%\lefteqn{T_N  V^{-1} = }  \\ 
$$
% \hspace{-.7cm}
\bbm (\lambda_1 I -M)^{-1} N S_2^T K_1 \hat{q}_1 & \hspace{-2mm}\ldots \hspace{-2mm} & ( \lambda_{n_1} I -M)^{-1}N S_2^T K_1 \hat{q}_{n_1} \ebm.
 $$
%\end{eqnarray*}
Note that the matrices $\lambda_i I -M$ are indeed invertible since $M$ is stable and the eigenvalues $\lambda_i$ of $A_1$ satisfy $|\lambda_i |  \geq 1$. Since, in addition, $(D_2 + EK_2)T_N = 0$, we see that for all $i = 1, \ldots ,n_1$ we have
$$
(D_2 + EK_2)(\lambda_1 I -M)^{-1} N S_2^T K_1 \hat{q}_i =0.
$$ 
Since $S_2^T K_1 \neq 0$, there must exist an index $i$ such that $S_2^TK_1 \hat{q}_i  \neq 0$.
For this $i$, let $z$ be a real vector such that $z^T S_2^T K_1 \hat{q}_i \neq 0$. Now choose $N := e_j z^T$, where $e_j$ denotes the $j$th standard basis vector in $\mathbb{R}^{n_2}$. By the discussion above we obtain
$
(D_2 + EK_2) (\lambda_1 I - M)^{-1} e_j = 0.
$
Since this holds for any $j$, we actually find $(D_2 + EK_2) (\lambda_1 I - M)^{-1} =0$, so $D_2 + EK_2 = 0$. Using \eqref{e:output}, we must also conclude that $D_1 + EK_1 =0$, which implies $\im D_1 \subseteq \im E$. Since $K_2$ is stabilizing it must be of the form $U_- X_{2-}^{\dagger}$ for some right-inverse $X_{2-}^{\dagger}$. This implies that $(X_{2+} - A_3 X_{1-}) X_{2-}^\dagger$ is stable and $D_2 +E U_-X_{2-}^\dagger =0$, that is, condition 1) holds.
This completes the proof of Theorem \ref{th:mainresult}. 
\end{proof}
\begin{remark}
In order to avoid  technicalities, in Theorem \ref{th:mainresult} we have assumed that the matrix $A_1$ is diagonalizable. The theorem however also holds if we drop this assumption. We omit the proof here. 
\end{remark}

\begin{remark}
According to Theorem 8, the data are informative for regulator design if and only if at least one of the conditions 1) or 2) holds. Condition 2) is in terms of solvability of the `data driven regulator equations' (15a) and (15b). These equations hold for all $(A_2, B_2)$ compatible with the data. In the end a matrix $T$ is defined as $T : = X_{2-}W$ and together with $V: = U_-W$ the classical regulator equations 
(14) are then satisfied for all $(A_2, B_2)$ compatible with the data. This is then `the classical design', and the difference $x_2(t) - Tx_1(t)$ converges to 0 as $t$ runs off to infinity (see \cite{Trentelman2001}, page 199)

If Condition 2) does not hold, but instead Condition 1) holds, then the only way to get output regulation is to make the entire output $z = (D_1 +EK_1) x_1 + (D_2 + E K_2) x_2$ equal to 0 {\em pointwise}. This is done by making $D_1 +EK_1 =0$ (possible because ${\rm im} ~D_1 \subseteq {\rm im}~E$\;) and $D_2 + E K_2 = 0$, where $K_2= U_- X_{2-}^{\dagger}$ also makes the system endo-stable.
\end{remark}

Note that Theorem \ref{th:mainresult} gives a characterization of all data that are
informative for regulator design, and gives a method to design a suitable regulator.
Nonetheless, the procedure to compute
this regulator is not entirely satisfactory. Indeed, in the case that condition 2) holds it is not
clear how to find a right inverse of $X_{2-}$  such that $(X_{2+} - A_3 X_{1-}) X_{2-}^\dagger$ is stable. In the case of condition 1), the additional constraint $D_2 +E U_-X_{2-}^\dagger =0$ needs to be satisfied. In general, $X_{2-}$ has many right inverses, and $(X_{2+} - A_3 X_{1-}) X_{2-}^\dagger$ can be stable, with or without $D_2 +E U_-X_{2-}^\dagger =0$, depending on the choice of the particular right inverse $X_{2-}^\dagger$.
To deal with this problem and to solve the problem of regulator design, we formulate the problem of finding a suitable right inverse in terms of feasibility of linear matrix inequalities (LMI's).

\begin{theorem} \label{th:regulator-with-LMI's}
Let $(U_-,X_{1-},X_2)$ be given data. Then the following hold:
\begin{enumerate}
\item
$X_{2-}$ has full row rank and has a right inverse $X_{2-}^{\dagger}$ such that  $(X_{2+} - A_3 X_{1-}) X_{2-}^\dagger$ is stable if and only if there exists a matrix $\Theta \in \mathbb{R}^{T \times n}$ such that 
\begin{equation} \label{e:regulator-sym}
X_{2-} \Theta = (X_{2-} \Theta)^\top 
\end{equation}
and
\begin{equation} \label{e:regulator-LMI}
 \bbm X_{2-} \Theta & (X_{2+} - A_3 X_{1-})\Theta \\
                                                                          \Theta^\top (X_{2+} - A_3 X_{1-})^\top & X_{2-} \Theta \ebm > 0.
\end{equation}    
\item $X_{2-}$ has full row rank and has a right inverse $X_{2-}^{\dagger}$ such that  $(X_{2+} - A_3 X_{1-}) X_{2-}^\dagger$ is stable with, in addition, $D_2 +E U_-X_{2-}^\dagger =0$ if and only if there exists a solution $\Theta \in \mathbb{R}^{T \times n}$ of \eqref{e:regulator-sym} and \eqref{e:regulator-LMI} that satisfies the linear equation $$(D_2X_{2-} + EU_-)\Theta = 0.$$     
\end{enumerate}
In both cases, a suitable right-inverse is given by $X_{2-}^{\dagger}: =   \Theta (X_{2-}\Theta)^{-1}$.                                                 
\end{theorem}
\begin{proof}
The proof can be given by adapting the proof of Theorem 17 in \cite{vanWaarde2020}.
\end{proof}
\begin{example}
We will now apply Theorem \ref{th:mainresult} to Example \ref{ex: example}. Putting the example in our general framework we have 
$$
\bmx_1 = \begin{bmatrix} \bmr_1 \\ \bmr_2 \\\bmd \end{bmatrix}, ~~\bmx_2 = \bmx,~~
A_1 = \begin{bmatrix} 0 & 1 & 0 \\ 
                                    -1 & 0  & 0 \\
                                    0 & 0 & 1
           \end{bmatrix}, 
$$                         
$$
A_3 = \begin{bmatrix} 0 & 0 & 1 
           \end{bmatrix},~~
D_1 = \begin{bmatrix} 1  & 0 & 0 \end{bmatrix}, ~~ D_2 = -1, ~~ E = 0.
$$
Assume $\tau = 3$, and the data on the disturbance input are $D_- = \begin{bmatrix} d(0) & d(1) & d(2) \end{bmatrix} =  \begin{bmatrix}  \frac{1}{2} &  \frac{1}{2} &   \frac{1}{2} \end{bmatrix}$. Since the signal to be tracked is $\cos \frac{1}{2} \pi t$, we must have $\bmr_1(0) = 1$, $\bmr_2(0) = 0$ so $\bmr_1(t) =  \cos \frac{1}{2} \pi t$ and $\bmr_2(t) = \cos \frac{1}{2} \pi(t + 1)$. This leads to 
$$
X_{1-} = \begin{bmatrix} r_1(0) & r_1(1) & r_1(2) \\
                                   r_2(0) & r_2(1) & r_2(2) \\
                                   d(0) & d(1) & d(2) \end{bmatrix} = 
         \begin{bmatrix}  1 & 0  & -1 \\
                                   0 & -1 & 0 \\
                                   \frac{1}{2} &  \frac{1}{2} &  \frac{1}{2} \end{bmatrix}.
$$
Assume that 
$
U_- =  \begin{bmatrix} u(0) & u(1) & u(2) \end{bmatrix} =  \begin{bmatrix} 1  & 0 &  0 \end{bmatrix}
$
and 
$
X_2 =  \begin{bmatrix} x_2(0) & x_2(1) & x_2(2) & x_2(3) \end{bmatrix} =  \begin{bmatrix}  0 &  \frac{3}{2}  & 2  &   \frac{5}{2} \end{bmatrix}.
$
It can be checked that condition 2) of Theorem \ref{th:mainresult} holds. Indeed, a solution $W$ to the linear equations \eqref{e:W} is given by
$$
W =  \begin{bmatrix} -1  & 1 & -1 \\
                                  \frac{2}{3} & 0   &  0 \\
                                   0  &  0  &  0 \end{bmatrix}.
$$   
Furthermore, $X_{2-}^{\dagger}  =   \begin{bmatrix}  -\frac{1}{2} &  \frac{2}{3}  & 0   \end{bmatrix}^T$ is a right-inverse of $X_{2-}$ and $(X_{2+} - A_3 X_{1-}) X_{2-}^\dagger = \frac{1}{2}$ is stable.  A regulator is then given by
$K_1 =  U_- (I - X_{2-}^\dagger X_{2-}) W =  \begin{bmatrix}  -\frac{1}{2} &  1  & -1   \end{bmatrix}$ and $K_2 := U_- X_{2-}^\dagger = -\frac{1}{2}$.

%Note that the above computation of a regulator uses only the given data. Also note that, due to the particular form of $A_3$, in this example only  exosystem data $\begin{bmatrix} d(0) & d(1) & d(2) \end{bmatrix}$ on the disturbance are actually used. 

It can be checked that the above data are compatible with the true endosystem $a_s = 1, b_s =1$. In fact, in this particular example, the true system is uniquely determined by the data. Indeed, this follows from the fact that  
$$
X_{2+} = \begin{bmatrix} a_s & b_s \end{bmatrix} \begin{bmatrix} X_{2-} \\ U_-  \end{bmatrix} + D_-~,
$$
in which 
$
\begin{bmatrix} X_{2-} \\ U_-  \end{bmatrix}
$
has full row rank. Thus, a regulator could also have been computed directly from the regulator equations \eqref{e:reg-eq} after first identfying the true endosystem $a_s = 1, b_s =1$. It can indeed be verified that $T=  \begin{bmatrix} 1 & 0 & 0 \end{bmatrix}$ together with $V = \begin{bmatrix} -1 & 1 & -1 \end{bmatrix}$ satisfy the regulator equations \eqref{e:reg-eq} for the true endosystem. By choosing $K_2 =  -\frac{1}{2}$, this would then lead to the same regulator as above with $K_1 = -K_2T + V = \begin{bmatrix}  -\frac{1}{2} &  1  & -1   \end{bmatrix}$. 
\end{example}
We note that, in general, the true endosystem may {\em not} be uniquely determined by the data. 
%In that case the option to first identify the true endosystem from the data and apply model based design is not available. 
This is illustrated by the following example.
\begin{example} \label{ex:nonid}
Consider the two-dimensional endosystem
$$
\bmx_2(t+1) = A_{2s}\bmx_2(t) + B_{2s}\bmu(t) + \begin{bmatrix} 0 \\ 1 \end{bmatrix} \bmd(t),
$$
where $A_{2s}$ and $B_{2s}$ are unknown $2\times 2$  and $2 \times 1$ matrices, respectively. Let 
$\bmx_2 = \begin{bmatrix} \bmx_{21} & \bmx_{22} \end{bmatrix}^T$. The disturbance input $\bmd$ is assumed to be a constant signal with finite amplitude, so is generated by $\bmd(t+1) = \bmd(t)$. We want to design a regulator so that $2\bmx_{21} + \frac{1}{2} \bmx_{22}$ tracks a given reference signal.
%for any incoming constant disturbance input. 
 In this example, the reference signals $\bmr$ are assumed to be generated by a given autonomous linear system with state space dimension, say, $n_1$. Its representation will be irrelevant here. The total exosystem will then have state space dimension $n_1 + 1$, and our output equation is given by
$
\bmz(t) = D_1 \bmx_1(t) + D_2 \bmx_2(t) + E \bmu(t),
$
with $D_1$ a $1 \times (n_1 + 1)$ matrix such that $D_1\bmx_1 = - \bmr$ and $D_2 = \bbm 2 & \frac{1}{2} \ebm$. We take $E = 2$. Also note that 
$
A_3 = {\small \bbm  0_{1 \times n_1}  &  0 \\
                    0_{1 \times n_1}  & 1   \ebm.}
$
Here, $0_{1 \times n_1}$ denotes $1 \times n_1$ zero matrix.  Suppose that $\tau = 2$ and assume we have the following data:
$$
U_- = \bbm -1 & -1 \ebm, ~D_-  =  \bbm  1 &  1    \ebm,~
X_2   = \bbm 1 & \frac{1}{2} & -\frac{1}{4} \\
                        0  &      2         &   \frac{5}{2}   \ebm. 
$$
These data can be seen to be generated by the true endosystem
$
A_{2s} = {\small \bbm 2 & ~ \frac{1}{8} \vspace{.6mm}\\ 
                 4          & ~ \frac{5}{4} \ebm,} ~ 
B_{2s} = {\small \bbm \frac{3}{2} \vspace{.6mm}\\ 3 \ebm}.
$
We now check condition 1) of Theorem \ref{th:mainresult}. First note that, indeed, $\im D_1 \subseteq \im E$.  Also, $X_{2-}$ is non-singular and 
%its inverse is 
%$
%X_{2-}^{-1} = \bbm 1  &  -\frac{1}{4} \\ 0   &   \frac{1}{2} \ebm.
%$
%Thus 
$
(X_{2+} - A_3 X_{1-}) X_{2-}^{-1} = {\small \bbm  \frac{1}{2}   & -\frac{1}{4} \\
                                                                    1                &  \frac{1}{2} \ebm}.
$
This matrix has eigenvalues $\frac{1}{2} \pm \frac{1}{2}i$, so is stable. Finally,
$D_2 +E U_-X_{2-}^{-1} =0$. According to Theorem \ref{th:mainresult}, a regulator for all endosystems compatible with the given data is given by
\begin{equation} \label{e:reg1}
K_2 = U_1 X_{2-}^{-1} = \bbm -1 & -\frac{1}{4} \ebm,~ K_1 = - \frac{1}{2} D_1.
\end{equation}
It can be verified that the set of endosystems compatible with our data is equal to the affine set
$$
\Sigma_{\calD} = \{ \left( \bbm a & \frac{1}{4} a - \frac{3}{8} \vspace{.6mm}\\
                                        b  & \frac{1}{4} b + \frac{1}{4} \ebm, 
                                        \bbm a - \frac{1}{2}  \vspace{.6mm} \\ b -1 \ebm \right )\mid a,b \in \mathbb{R} \}.
$$
The controller given by \eqref{e:reg1} is a regulator for all these endosystems.
\end{example}
\begin{remark}
It is also possible to consider the situation that, in addition to $A_2$ and $B_2$, also the matrix $A_3$ (representing how the exosignal $\bmx_1$ enters the endosystem) is unknown. In that case, the set all endosystems compatible with the data $(U_-,X_2,X_-)$ is defined as follows:
\[
\Sigma_{D} = \{(A_2,B_2,A_3) \mid X_{2+} = A_2 X_{2-} + B_2U_- + A_3 X_{1-} \}.
\]
The data are then called informative for regulator design if there exists a single regulator $u = K_1 x_1 + K_2 x_2$ for all endosystems in $\Sigma_D$.  
The analogue of Theorem \ref{th:mainresult} for this situation is as follows. Both In Conditions 1) and 2), an additional condition  $X_{1-}X_{2-}^\dagger =0$ should be imposed on a suitable right-inverse of $X_{2-}$. In addition, in Condition 2), the old data-driven regulator equations \eqref{e:W} should be replaced by:
\begin{subequations} \label{e:Wnew}
\begin{align}
X_{2-}W A_1 - X_{2+} W = 0 ,\\
X_{1-} W = I, \\
D_1 +(D_2 X_{2-} + E U_-) W  = 0.
\end{align} 
\end{subequations}
Note that, as expected, $A_3$ no longer appears in the equations (it is unknown). In both cases, the formulas for $K_1$ and $K_2$ are the same as in Theorem \ref{th:mainresult}. Due to space limitations, the proof is omitted.
%Obviously, a necessary condition is {\em informativity for endo-stabilization}: $A_2 + B_2 K_2$ should be stable for all $(A_2,B_2)$ for which there exists $A_3$ such that $(A_2,B_2,A_3) \in \Sigma_D$. It was shown in \cite{vanWaarde2020c} that such $K_2$ exists if and only if $X_{2-}$ has full row rank and it has a right-inverse $X_{2-}^\dagger$ such that  $X_{2+} X_{2-}^\dagger$ is stable and $X_{1-} X_{2-}^\dagger = 0$. Finding necessary and sufficient conditions for informativity for regulator design is still an open problem and is left for future research.
\end{remark}

\section{Conclusions}              
We have introduced the notion of data informativity in the context of  the classical algebraic regulator problem.  Our main results are necessary and sufficient conditions for a given set of data to be informative for regulator design, and formulas to compute regulators using only this set of data. We have recast the computation of suitable regulators in terms of feasibility of LMI's. Our results have been illustrated by means of two extended examples. In the present paper, only static state feedback regulators have been considered. As an open problem for future research we mention the extension to dynamic output feedback regulators. Results obtained in \cite{vanWaarde2020} on the problem of stabilization by dynamic output feedback (both in terms of input-state-output data and input-output data) are expected to be relevant here. 
Another possible venue for future research is to consider the situation that, in addition to $A_2$, $A_3$ and $B_2$, also the matrix $A_3$ is unknown. Finally, it would be interesting to  include noise in the problem formulation, and to consider the situation in which, in addition to the modeled disturbances, bounded noise may enter the unknown endosystem (see also \cite{vanWaarde2020b}).
%{\color{blue}A topic for future research might also be to develop a data driven theory of linear quadratic tracking, where reference signals may be arbitrary, and not necessarily generated by a linear exosystem.}

\vspace{-3mm}
\bibliographystyle{plain}
\bibliography{references.bib}

\end{document}

%% file: ka-newcommands.tex
\DeclareMathOperator{\im}{im}

\DeclareMathOperator{\diag}{diag}

\let\geq\geqslant

\newcommand{\calD}{\ensuremath{\mathcal{D}}}

\newcommand{\bmat}{\begin{matrix}}
\newcommand{\emat}{\end{matrix}}
\newcommand{\bbm}{\begin{bmatrix}}
\newcommand{\ebm}{\end{bmatrix}}
\newcommand{\bbma}{\begin{bmatrix*}[r]}
\newcommand{\ebma}{\end{bmatrix*}}
\newcommand{\bpm}{\begin{pmatrix}}
\newcommand{\epm}{\end{pmatrix}}
\newcommand{\bvm}{\begin{vmatrix}}
\newcommand{\evm}{\end{vmatrix}}
\newcommand{\bse}{\begin{subequations}}
\newcommand{\ese}{\end{subequations}}
\newcommand{\beq}{\begin{equation}}
\newcommand{\eeq}{\end{equation}}
\newcommand{\ben}{\renewcommand{\labelenumi}{\arabic{enumi}.}
\renewcommand{\theenumi}{\arabic{enumi}}\begin{enumerate}}

\newcommand{\een}{\end{enumerate}}

\newcommand{\beni}{\renewcommand{\labelenumi}{\roman{enumi}.}
\renewcommand{\theenumi}{\roman{enumi}}\begin{enumerate}}

\newcommand{\eeni}{\end{enumerate}}

\newcommand{\bena}{\renewcommand{\labelenumi}{\alph{enumi}.}
\renewcommand{\theenumi}{\alph{enumi}}\begin{enumerate}}

\newcommand{\eena}{\end{enumerate}}

\newcommand{\bit}{\begin{itemize}}
\newcommand{\eit}{\end{itemize}}
\newcommand{\bthe}{\begin{theorem}}
\newcommand{\ethe}{\end{theorem}}
\newcommand{\blem}{\begin{lemma}}
\newcommand{\elem}{\end{lemma}}
\newcommand{\bprop}{\begin{proposition}}
\newcommand{\eprop}{\end{proposition}}
\newcommand{\bex}{\begin{example}}
\newcommand{\eex}{\end{example}}
\newcommand{\bas}{\begin{assumption}}
\newcommand{\eas}{\end{assumption}}
\newcommand{\bre}{\begin{remark}}
\newcommand{\ere}{\end{remark}}
\newcommand{\bcor}{\begin{corollary}}
\newcommand{\ecor}{\end{corollary}}
\newcommand{\bdfn}{\begin{definition}}
\newcommand{\edfn}{\end{definition}}
\newcommand{\bcon}{\begin{conjecture}}
\newcommand{\econ}{\end{conjecture}}